\documentclass[11pt]{article}

\usepackage{amssymb, amsmath, amsthm, graphicx}
\usepackage[left=1in,top=1in,right=1in]{geometry}
\usepackage{algorithm}
\usepackage{algpseudocode}

\date{}

\theoremstyle{plain}
      \newtheorem{theorem}{Theorem}[section]
      \newtheorem{lemma}[theorem]{Lemma}
      \newtheorem{claim}[theorem]{Claim}
      
      \newtheorem{observation}[theorem]{Observation}
      \newtheorem{corollary}[theorem]{Corollary}

\theoremstyle{definition}

\theoremstyle{remark}

\title{Extremal structure in dense arrangements of $k$-intersecting curves}

\begin{document}

\author{Andrew Suk\thanks{Department of Mathematics, University of California at San Diego, La Jolla, CA, 92093 USA. Supported by NSF grants DMS-1952786 and DMS-2246847. Email: {\tt asuk@ucsd.edu}.}  \and Su Zhou\thanks{Department of Mathematics, University of California at San Diego, La Jolla, CA, 92093 USA. Supported by NSF grants DMS-1952786 and DMS-2246847. Email: {\tt suzhou@ucsd.edu}.} }

\maketitle

\begin{abstract}
Let $P$ be a set of $n$ points in the plane, and let $\mathcal C$ be a collection of $n$ simple $k$-intersecting curves, meaning that every two distinct curves of $\mathcal C$ meet in at most $k$ points. A classical theorem of Pach and Sharir from 1998 gives the upper bound $I(P,\mathcal C)=O_k(n^{(3k+1)/(2k+1)})$.  We prove that this bound can be improved when one excludes a complete
local incidence pattern. More precisely, for any fixed integers $s>k+1\ge2$, if there do not exist $s$ points of $P$ such that every $(k+1)$-tuple among them is contained in a distinct curve of $\mathcal C$, then
\[
I(P,\mathcal C)=o\!\left(n^{\frac{3k+1}{2k+1}}\right).
\]
In the special case of pseudo-segments, this extends Solymosi’s theorem on dense point-line arrangements to dense arrangements of pseudo-segments.
\end{abstract}

\section{Introduction}
Let $P$ be a set of $m$ points and let $\mathcal{L}$ be a set of $n$ lines in the plane. 
An \emph{incidence} is a pair $(p,\ell)\in P\times \mathcal{L}$ such that $p\in \ell$.  
We denote by $I(P,\mathcal{L})$ the number of incidences between $P$ and $\mathcal{L}$, that is,
\[
I(P,\mathcal{L})
=
\bigl|\{(p,\ell)\in P\times \mathcal{L} : p \in \ell\}\bigr|.
\]
A classical theorem of Szemer\'edi and Trotter~\cite{ST} states that for any set $P$ of $m$ points and any set $\mathcal{L}$ of $n$ lines in the plane, we have
\begin{equation}\label{IPL}
    I(P,\mathcal{L}) = O\!\left(m^{2/3}n^{2/3} + m + n\right).
\end{equation}

\noindent
Moreover, this bound is asymptotically tight, as shown by the
standard integer-lattice construction.  The Szemer\'edi--Trotter theorem is one of the most central results in discrete geometry and has served as a fundamental tool with far-reaching applications in combinatorics, discrete geometry, additive number theory, and theoretical computer science \cite{ST, PachSharir, MatousekBook, SharirAgarwal95, Elekes97, GuthKatz}.  

It is widely believed that extremal configurations maximizing the number of incidences between $n$ points and $n$ lines in the plane exhibit an underlying lattice-type structure. 
An early result in this direction is due to Solymosi~\cite{SolymosiLocallyDense}, who proved that if $P$ is a set of $n$ points and $\mathcal{L}$ is a set of $n$ lines in the plane whose arrangement avoids a configuration consisting of $s$ points together with a distinct line of $\mathcal{L}$ through every pair, then $I(P,\mathcal{L}) = o(n^{4/3}).$  See \cite{sukm,SuT,BF} for further results in this direction.

In 1997, Sz\'ekely~\cite{Sz} gave an elegant proof of the Szemer\'edi--Trotter theorem via the crossing lemma. 
His approach has since become a standard tool in incidence geometry and, in particular, implies that the Szemer\'edi--Trotter bound extends to families of \emph{pseudo-segments}, that is, collections of planar curves in which any two intersect in at most one point.

Building on Sz\'ekely's crossing-lemma method, Pach and Sharir~\cite{PachSharir}
established general incidence bounds for families of plane curves with
bounded degrees of freedom.  Let \(\mathcal C\) be a family of simple curves in the plane. We say that \(\mathcal C\) is \emph{\(k\)-intersecting} if every two distinct curves of \(\mathcal C\) meet in at most \(k\) points.  Their result implies the following.

\begin{theorem}[Pach--Sharir, \cite{PachSharir}]\label{PS}
Let $k \ge 1$ be an integer, and let $P$ be a set of $n$ points in the plane and $\mathcal C$ be a finite collection of $n$ $k$-intersecting curves. Then there exists a constant $c_k>0$, depending only on $k$, such that
\[
I(P,\mathcal C)\le c_k n^{\frac{3k+1}{2k+1}}.
\]
\end{theorem}

Our main result shows that the Pach--Sharir bound is not asymptotically tight under a natural local forbidden-configuration hypothesis.

\begin{theorem}\label{soly}
Fix \(s>k+1\ge 2\), and let \(P\) be a set of \(n\) points in the plane and \(\mathcal C\) a set of \(n\) simple \(k\)-intersecting curves.
If there do not exist \(s\) points of \(P\) such that every \((k+1)\)-tuple among them is contained in a distinct curve of \(\mathcal C\), then
\[
I(P,\mathcal C)=o\!\left(n^{\frac{3k+1}{2k+1}}\right).
\]
\end{theorem}

\noindent
Our proof is based on a refined $r$-division argument. In particular,
we prove a modified $r$-division theorem adapted to a prescribed set of
vertices. This allows us to work directly with general
$k$-intersecting curves, rather than lines or algebraic curves of
bounded complexity.

As an immediate corollary, we extend Solymosi's result~\cite{SolymosiLocallyDense} from lines to pseudo-segments. 

\begin{corollary}
Fix \(s>2\), and let \(P\) be a set of \(n\) points in the plane and \(\mathcal C\) a set of \(n\) pseudo-segments.
If there do not exist \(s\) points of \(P\) such that every pair among them is contained in a distinct curve of \(\mathcal C\), then
\[
I(P,\mathcal C)=o(n^{4/3}).
\]
\end{corollary}

In the case of points and lines, the first author and Tomon \cite{SuT} constructed point-line arrangements with many incidences while forbidding several natural dense subconfigurations, including grids, fans, and short cycles.  In particular, they proved that there exist \(n\) points \(P\) and \(n\) lines \(\mathcal L\) in the plane with no three points of \(P\) such that every pair of them is contained in a distinct line of \(\mathcal L\), and which determine at least \(n^{7/6-o(1)}\) incidences.  We improve this bound to \(n^{6/5-o(1)}\) in the pseudo-segment setting, and more generally prove the following result.

\begin{theorem}\label{thm:lower}
For every fixed integer \(s\ge 3\), there exists a set \(P\) of \(n\) points and a set \(\mathcal C\) of \(n\) pseudo-segments in the plane such that there do not exist \(s\) points of \(P\) with the property that every pair of them lies on a distinct curve of \(\mathcal C\), and
\[
I(P,\mathcal C)\ge
n^{\frac43-\frac{s-1}{3(s^2-s-1)}-o(1)}.
\]
\end{theorem}

It is not known whether the Pach--Sharir bound in Theorem~\ref{PS} is asymptotically tight for any fixed \(k\ge2\). The best lower bound remains the classical Szemer\'edi--Trotter point--line construction, which gives \(I(P,\mathcal C)=\Omega(n^{4/3})\). Thus, for every \(k\ge2\), there remains a substantial gap between the best known upper and lower bounds.  For incidences between points and algebraic curves of bounded degree in the plane, see the work of Sharir and Zahl~\cite{ShZ}.

Our paper is organized as follows. In Section~\ref{sec:rdiv}, we prove a refined \(r\)-division theorem adapted to our setting. In Section~\ref{sec:dense}, we combine this decomposition theorem with the hypergraph removal lemma to prove Theorem~\ref{soly} and its non-diagonal variant. In Section~\ref{sec:lower}, we prove the lower bound construction in Theorem~\ref{thm:lower}.  We conclude with some final remarks.

\section{A refined $r$-division theorem}\label{sec:rdiv}

A classical result due to Lipton and Tarjan~\cite{LiptonTarjan1979} from 1979, known as the \emph{planar separator theorem}, states that every $N$-vertex planar graph contains a set of at most $O(\sqrt{N})$ vertices whose removal leaves no connected component with more than $2N/3$ of the vertices.
It has played a central role in the structural theory of planar graphs, enabling efficient divide-and-conquer arguments.
Over the past several decades, numerous variants of the planar separator theorem have been established, including a separator theorem for nonplanar graphs due to Alon, Seymour, and Thomas~\cite{AlonSeymourThomas1990}, as well as the shallow-minor and excluded-minor separator results of Plotkin, Rao, and Smith~\cite{PlotkinRaoSmith1994}.
We will use the following simple-cycle version due to Miller.  Throughout this paper, all simple cycles $C$ will be viewed as a simple closed curve in the plane.  Moreover, all graphs are allowed to have multiple edges.

\begin{theorem}[Miller, \cite{Miller1986}]\label{cycle}
Let $G$ be a planar graph with $N$ vertices embedded in the plane with all faces of size 2 or 3, and let $w : V(G) \to \mathbb{R}_{\ge 0}$ be a nonnegative weight function on the vertices of $G$ with total weight $W$. Then there exists a simple cycle $C$ in $G$ such that $|C| = O(\sqrt{N})$, the  total weight of the vertices strictly inside $C$ is at most $2W/3$, and the total weight of the vertices strictly outside $C$ is at most $2W/3$.
 \end{theorem}

In 1987, Frederickson~\cite{Fre87} introduced the notion of an \(r\)-division of a planar graph, providing a refined separator-based decomposition into a structured partition.
To define \(r\)-divisions precisely, we first recall several definitions.  Let \(G=(V,E,F)\) be a planar graph embedded in the plane, where \(V\) is the set of vertices, represented by distinct points in the plane, \(E\) is the set of edges, drawn as simple curves connecting the corresponding pairs of vertices, and \(F\) is the set of faces determined by this embedding.
A subgraph \(R\subseteq G\) is called a \emph{region of \(G\)} if \(R\) is a connected edge-induced subgraph embedded in the plane.
A \emph{division} of \(G\) is a collection of regions \(\mathcal R=\{R_1,\dots,R_k\}\) such that each edge of \(G\) lies in at least one region \(R_i\).
Given a division \(\mathcal R\) of \(G\), a vertex \(v\in V(G)\) is called a \emph{boundary vertex} if it lies in more than one region of \(\mathcal R\).
For convenience, we denote by \(b(R)\subseteq V(R)\) the set of boundary vertices of a region \(R\).
We say that a division \(\mathcal R\) is an \emph{\(r\)-division} if each region \(R_i\) has at most \(O(r)\) vertices, \(|\mathcal R| = O(|V|/r)\), and each \(R_i\) has at most \(O(\sqrt r)\) boundary vertices.

\begin{theorem}[\cite{Fre87}]\label{fred}
    There is an absolute constant $r_0$ such that the following holds.  For $r\geq r_0$, every planar graph embedded in the plane admits an $r$-division.
\end{theorem}
Let us now briefly describe the proof of Theorem \ref{fred} presented by Klein, Mozes, and Sommer in \cite{KMS13}.  The proof is based on recursively applying planar cycle-separators.  For our purposes, we encode this recursive decomposition by a rooted binary tree.  Given a planar graph $G$ embedded in the plane with all faces of size 2 or 3, a decomposition tree $\mathcal T$ of $G$ is constructed recursively as follows.
Each node \(x\in\mathcal T\) is associated with a region \(R_x\subseteq G\).  
The root of \(\mathcal T\) is the entire graph \(G\).

Given a node \(x\in\mathcal T\) with corresponding region \(R_x\subseteq G\),  we define $b(R_x)$ to be the set of vertices of $R_x$ that already lie on some separator cycle chosen on the path from the root to $x$. Suppose that, for a sufficiently large constant \(c_0>0\), we have
\[
|V(R_x)|>c_0r
\qquad\text{or}\qquad
|b(R_x)|\ge c_0\sqrt r.
\]
We then define the children of \(x\) as follows.  
First, we triangulate every face of \(R_x\) of size greater than 3, by adding edges inside the face from a single vertex.  Let $R'_x$ be the resulting graph.  
Applying Miller's cycle separator theorem (Theorem~\ref{cycle}) to \(R'_x\), we obtain a simple cycle separator \(C_x\) chosen according to the following rule: if the depth of \(x\) in \(\mathcal T\) is \(\ell\), then \(C_x\) is selected to balance vertices of \(R_x\) when \(\ell\equiv 0\pmod 2\), and to balance boundary vertices of \(R_x\) when \(\ell\equiv 1\pmod 2\).  Let us remark that $V(C_x) \subset V(R_x)$, but some of the edges of $C_x$ may not be edges in region $R_x$.  The cycle \(C_x\) partitions the faces of \(R'_x\) into two parts. Accordingly, we define two child regions \(R_{x,0}\) and \(R_{x,1}\): the region \(R_{x,0}\) is the edge-induced subgraph of \(R_x\) consisting of all edges lying inside or on \(C_x\), while \(R_{x,1}\) is the edge-induced subgraph of \(R_x\) consisting of all edges lying outside or on \(C_x\).  These regions are assigned to the two children of \(x\) in \(\mathcal T\).  
Once the algorithm terminates, the regions corresponding to the leaves of \(\mathcal T\) form the desired \(r\)-division \(\mathcal R\).  

One useful structure in the construction above is the collection of simple cycle separators \(\mathcal S\) called throughout the recursion tree \(\mathcal T\).  
That is, let
\[
\mathcal S:=\{C_x:x\text{ is an internal node of }\mathcal T\}.
\]
Thus, the pair \((\mathcal T,\mathcal S)\) gives rise to the \(r\)-division \(\mathcal R\).  
Moreover, a vertex \(v\in V(G)\) is a boundary vertex in the division \(\mathcal R\) if and only if \(v\in V(C_x)\) for some separating cycle \(C_x\in\mathcal S\).  Indeed, if
\(v\) is a boundary vertex, then \(v\) lies in two distinct leaf regions of the
recursion tree.  Consider the two root-to-leaf paths leading to these leaves,
and let \(x\) be their last common node.  Then the two leaves lie in different
child subtrees of \(x\).  Since both leaves contain \(v\), both children of
\(x\) contain \(v\).  By the construction of the two children, this can happen
only if \(v\in V(C_x)\). Conversely, if \(v\in V(C_x)\) for some separating cycle \(C_x\in\mathcal S\), we may assume without loss of generality that $x$ is the unique highest node for which \(v\in V(C_x)\) holds. Then $v$ must lie in both children of $x$, and later division can only increase or preserve the multiplicity of $v$ counted in all regions. Hence, $v$ must be shared by at least two distinct leave regions of the recursion tree. This statement also holds for our refined $r$-division argument below.

  We now prove the following variant of Frederickson's $r$-division theorem, which will be used to prove Theorem \ref{soly}.

\begin{theorem}\label{rdiv}
There exists an absolute constant $r_0>0$ such that the following holds. Let $G$ be a planar graph on $N$ vertices embedded in the plane with all faces of size 2 or 3, and let $P \subseteq V(G)$.  Then for every $r \ge r_0$ and every $t \ge 1$, there exists a division $\mathcal R$ of $G$, obtained from a recursion tree $\mathcal T$ and a collection of simple cycles $\mathcal S$, such that
\[
|\mathcal{R}| = O\!\left(\frac{N}{r} + \frac{|P|}{t}\right),
\]
and every region $R \in \mathcal{R}$ satisfies $|V(R)| \le O(r),
|b(R)| = O(\sqrt r),$ and $|(V(R)\setminus b(R))\cap P| \le O(t).$  Moreover, every simple cycle $C_x \in \mathcal{S}$ satisfies one of the following:

\begin{enumerate}
     
    \item the number of vertices of $G$ inside or on $C_x$ is at least $r/8$, or

    \item the number of boundary vertices inside or on $C_x$ is at least $\sqrt{r}/8$, or

    \item the number of points from $P$ inside or on $C_x$ is at least $t/8.$
 \end{enumerate}
\end{theorem}

The proof of Theorem~\ref{rdiv} is a straightforward adaptation of the argument of Klein, Mozes, and Sommer~\cite{KMS13}.  We include the full proof here, which requires the following lemma (see Lemma 2 in \cite{KMS13}).

\begin{lemma}[\cite{KMS13}]\label{lem:KMS-recurrence}
Let \(1/2\le \beta<1\), let \(c\) be a constant, and let \(T_r(m)\) satisfy
\[
T_r(m)\le
\begin{cases}
\rho m^\beta+\max\limits_{\{\alpha_i\}}\sum_{i=1}^8 T_r(\alpha_i m), & m>r,\\[4pt]
0, & m\le r,
\end{cases}
\]
where the maximum is taken over all \(\alpha_1,\dots,\alpha_8\) such that
\[
\alpha_i\le \frac34+\frac{c}{\sqrt m}\qquad (i=1,\dots,8),
\]
and
\[
1\le \sum_{i=1}^8 \alpha_i \le 1+\frac{c}{\sqrt m}.
\]
Then there exists a constant \(s\) such that, for every \(r\ge s\),
\[
T_r(m)=O\!\left(\frac{m}{r^{1-\beta}}\right).
\]
\end{lemma}

\noindent Since we will be using Lemma \ref{lem:KMS-recurrence} verbatim, it will be more convenient to apply the cycle separator theorem, Theorem \ref{cycle}, that achieves a weaker 3/4 balance rather than 2/3.

\begin{proof}[Proof of Theorem \ref{rdiv}]
Let $G$ be a planar graph on $N$ vertices embedded in the plane with all faces of size 2 or 3, and let $P \subseteq V(G)$.  We construct a rooted binary tree \(\mathcal T\) whose nodes \(x\) correspond to connected embedded regions \(R_x\subseteq G\), with the root $\hat x$ corresponding to \(G\) itself. For each node \(x\), write
\[
n(x):=|V(R_x)|,\qquad
b(x):=|b(R_x)|,\qquad
p(x):=|(V(R_x)\setminus b(R_x))\cap P|.
\]

Let $c_0$ be a large absolute constant that will be determined later.  A node \(x \in \mathcal T\) becomes a leaf whenever all three of the following inequalities hold: 
\begin{equation}\label{eq:leaf-threshold-rdiv}
n(x)\le c_0r,\qquad
b(x)\le c_0\sqrt r,\qquad
p(x)\le c_0t.
\end{equation}
Otherwise, triangulate each face of \(R_x\) of size greater than three by adding noncrossing diagonals from a single vertex of that face, thereby obtaining a planar graph \(R_x'\) embedded in the plane with the same vertex set and every face of size \(2\) or \(3\).

Suppose \(x \in \mathcal T\) is a node at depth \(\ell\) that does not satisfy \eqref{eq:leaf-threshold-rdiv}. According to \(\ell \bmod 3\), we apply Theorem~\ref{cycle} with the following choice of weights:
\begin{itemize}
    \item if \(\ell\equiv 0 \pmod 3\) and \(n(x) > c_0r\), we assign weight \(1\) to every vertex of \(R_x\);
    \item if \(\ell\equiv 1 \pmod 3\) and \(b(x) > c_0\sqrt{r}\), we assign weight \(1\) to each boundary vertex in \(b(R_x)\) and \(0\) to all other vertices;
    \item if \(\ell\equiv 2 \pmod 3\) and \(p(x) > c_0t\), we assign weight \(1\) to each point of \((V(R_x)\setminus b(R_x))\cap P\) and \(0\) to all other vertices.
\end{itemize}

\noindent If the parameter prescribed by the congruence class of \(\ell\) is already below its threshold, that is, if \(n(x) \le c_0r\), \(b(x) \le c_0\sqrt r\), or \(p(x) \le c_0t\), then we instead choose any one of the remaining parameters that still exceeds its threshold and balance that parameter. Such a parameter must exist, since \(x\) does not satisfy \eqref{eq:leaf-threshold-rdiv}. Therefore, every internal node is split by a separator that balances a parameter still above threshold.  By Theorem~\ref{cycle}, there is an absolute constant \(c_1\) such that the resulting simple cycle separator \(C_x\) in $R'_x$ that satisfies
\[
|C_x|\le c_1\sqrt{n_x},
\]
and each side of \(C_x\) contains at most two-thirds of the balanced weight. We let \(R_{x,0}\) and \(R_{x,1}\) denote the two child regions consisting of the edges of \(R_x\) lying inside-or-on and outside-or-on \(C_x\), respectively.

\begin{observation}
    The graphs $R_{x,0}$ and $R_{x,1}$ are connected.
\end{observation}

\begin{proof}
Let \(R'_{x,0}\) be the subgraph of the triangulated region \(R'_x\) lying inside-or-on \(C_x\). Then \(R'_{x,0}\) is connected. Let \(u,v\in V(R_{x,0})\), and let \(\alpha\) be a walk from \(u\) to \(v\) in \(R'_{x,0}\).

We now remove the added diagonals from \(\alpha\). Whenever \(\alpha\) uses added diagonals inside a face \(f\) of \(R_x\), replace the portion of \(\alpha\) between its first and last vertices on \(\partial f\) by the corresponding boundary walk of \(f\) on the same side of \(C_x\). This is possible because \(f\) was triangulated as a fan from a single boundary vertex, so each side of the fan is bounded by added diagonals together with a boundary walk of \(f\). The replacement uses only edges of \(R_x\) lying inside-or-on \(C_x\).

After performing this replacement for every face of \(R_x\), we obtain a walk from \(u\) to \(v\) in \(R_{x,0}\). Hence \(R_{x,0}\) is connected.  Same argument follows for $R_{x,1}$.
\end{proof}

Let the division \(\mathcal R\) be the set of leaves of \(\mathcal T\), and define
\[
\mathcal S:=\{C_x:x\in\mathcal T \text{ is an internal node}\}.
\]

\noindent By construction, every region $R \in \mathcal R$ satisfies (\ref{eq:leaf-threshold-rdiv}).

Let \(v\) be any boundary vertex of the division \(\mathcal R\). Since \(v\) is a boundary vertex of the final division, it lies on some separator cycle used during the recursion. Choose the first such cycle \(C_x\in\mathcal S\) along a root-to-leaf path leading to a leaf containing \(v\).  Since \(x\) is an internal node, at least one of
\[
n(x)>c_0r,\qquad
b(x)>c_0\sqrt r,\qquad
p(x)>c_0t
\]
holds. By construction, \(C_x\) was chosen to balance one of the parameters still exceeding threshold. Hence each side of \(C_x\) contains at least one-third of that parameter. For $c_0$ sufficiently large, we have
\begin{enumerate}
    \item if \(C_x\) balances vertices, then each side of \(C_x\) contains at least \(n(x)/4\ge c_0r/4 \geq r/8\) vertices;
    \item if \(C_x\) balances boundary vertices, then each side of \(C_x\) contains at least \(b(x)/4\ge c_0\sqrt r/4\geq \sqrt{r}/4\) boundary vertices;
    \item if \(C_x\) balances points of \(P\), then each side of \(C_x\) contains at least \(p(x)/4\ge c_0t/4 \geq t/8\) such points.
\end{enumerate}

Finally, we show that $|\mathcal R| \leq O\!\left(\frac{N}{r} + \frac{|P|}{t}\right)$.  Here, we follow the outline of the proof in Klein, Mozes, and Sommer in \cite{KMS13}.  For a node \(x\) of \(\mathcal T\) and a set \(S\) of descendants of \(x\) with no ancestor-descendant pair in $S$, define
\[
L(x,S):=-\,n(x)+\sum_{y\in S} n(y).
\]
Roughly speaking, \(L(x,S)\) is the total overcount of vertices when \(R_x\) is replaced by the regions \(R_y\), \(y\in S\). For each node $x \in \mathcal T,$ let \(S_r(x)\) denote the set of descendants \(y\) of \(x\) such that $
n(y)\le c_0r$ and the parent of \(y\) has more than \(c_0r\) vertices.  Clearly, no two nodes in $S_r(x)$ are an ancestor-descendant pair.

If \(x_0,x_1\) are the two children of a node \(x\), then regardless of what $C_x$ balances, we have
\begin{align}
n(x_0)+n(x_1) &\le n(x)+c_1\sqrt{n(x)}, \label{eq:n-split-final}\\
b(x_0)+b(x_1) &\le b(x)+c_1\sqrt{n(x)}, \label{eq:b-split-final}\\
p(x_0)+p(x_1) &\le p(x). \label{eq:p-split-final}
\end{align}
Indeed, the two child regions overlap only along the separator cycle \(C_x\), whose length is at most
\(c_1\sqrt{n(x)}\), and every point of \((V(R_x)\setminus b(R_x))\cap P\) belongs to the interior of at
most one child. Hence,

\begin{equation}\label{app}
   L(x,\{x_0,x_1\}) \leq c_1\sqrt{n(x)}. 
\end{equation}

\noindent We now make the following claim.

\begin{claim}\label{eq:L-root-Sr-final}
    Let $\hat x$ be the root of $\mathcal T$.  Then there is an absolute constant $c_2 > 0$ such that 

    $$L(\hat x, S_r(\hat x)) \leq c_2\frac{N}{\sqrt{r}}.$$
\end{claim}

\begin{proof}

 For an integer \(m>c_0r\), define
\[
B_r(m):=\max\{L(x,S_r(x)): x\in\mathcal T,\ n(x)\le m,\ \ell(x)\equiv 0\pmod 3\}.
\]

\noindent We claim that
\begin{equation}\label{eq:Br-recurrence-final}
B_r(m)\le 7c_1\sqrt m+\max_{\{\alpha_i\}}\sum_{i=1}^8 B_r(\alpha_i m),
\end{equation}
where the maximum is over all \(\alpha_1,\dots,\alpha_8\) satisfying
\begin{equation}\label{eq:alpha-upper-final}
\alpha_i\le \frac34+\frac{c_1}{\sqrt m}
\qquad (i=1,\dots,8),
\end{equation}
and
\begin{equation}\label{eq:alpha-sum-final}
1\le \sum_{i=1}^8\alpha_i\le 1+\frac{7c_1}{\sqrt m}.
\end{equation}

To prove this, fix a node \(x\) with \(\ell(x)\equiv 0\pmod 3\) and \(c_0r<n(x)\le m\) such that $B_r(m) = L(x,S_r(x))$.  
Let \(y_1,\dots,y_k\) be the rootmost descendants \(y\) of \(x\) such that $n(y)\le c_0r$ or $\ell(y)-\ell(x)=3.$  Clearly \(k\le 8\).
A repeated application of (\ref{app}) gives
\[
L(x,\{y_1,\dots,y_k\})\le (k-1)c_1\sqrt{n(x)}\le 7c_1\sqrt m.
\]
Hence
\begin{equation}\label{eq:sum-yi-final}
\sum_{i=1}^k n(y_i)\le n(x)+7c_1\sqrt{n(x)}\le m+7c_1\sqrt m.
\end{equation}
For \(1\le i\le k\), define \(\alpha_i\) by \(\alpha_i m=n(y_i)\), and for \(k<i\le 8\), set \(\alpha_i=0\).
Then \eqref{eq:sum-yi-final} implies the upper bound in \eqref{eq:alpha-sum-final}.
The lower bound in \eqref{eq:alpha-sum-final} holds because every vertex of \(R_x\) belongs to at least one
of the regions \(R_{y_i}\).

Now let \(x_0,x_1\) be the two children of \(x\). Since \(\ell(x)\equiv 0\pmod 3\) and \(n(x)>c_0r\), the
separator chosen at \(x\) balances the vertices, and therefore
\[
\max\{n(x_0),n(x_1)\}\le \frac34 n(x)+c_1\sqrt{n(x)}\le \frac34 m+c_1\sqrt m.
\]
Each region \(R_{y_i}\) is contained in one of \(R_{x_0},R_{x_1}\), so
\[
n(y_i)\le \frac34 m+c_1\sqrt m,
\]
which implies \eqref{eq:alpha-upper-final}. Without loss of generality, there is an index \(0\le k'\le k\) for which every region corresponding to \(y_1,\ldots,y_{k'}\) has fewer than \(c_0r\) vertices, and every region corresponding to \(y_{k'+1},\ldots,y_k\) has more than \(c_0r\) vertices.  Hence, each node $y_{k'+1},\ldots , y_k$ are at level $0  \pmod 3 $ in the recursion tree $\mathcal{T}$. Therefore, we have
\begin{align*}
B_r(m)
&= L(x,S_r(x)) \\
&\le
L\bigl(x,\{y_1,\ldots,y_{k'}\}\bigr)
+\sum_{i=k'+1}^{k} L(y_i,S_r(y_i)) \\
&\le
7c_1\sqrt m+\sum_{i=1}^{8} B_r(\alpha_i m).
\end{align*}
This proves \eqref{eq:Br-recurrence-final}.  Since \(B_r(m)=0\) for \(m\le c_0r\), Lemma~\ref{lem:KMS-recurrence} applies to
\eqref{eq:Br-recurrence-final}. Hence there exist absolute constants \(r_0\) and \(c_2\), depending only on
\(c_0\) and \(c_1\), such that for every \(r\ge r_0\), $B_r(m)\le c_2 m/\sqrt r $ for all $m \geq 1$. Applying this with \(m=N=n(\hat x)\), we obtain $L(\hat x,S_r(\hat x))\le c_2N/\sqrt{r}.$
\end{proof}

By Claim \ref{eq:L-root-Sr-final}, the total overcount of vertices among the regions in $S_r(\hat x)$ is $O(N/\sqrt{r})$.  Therefore
\[
\sum_{x\in S_r} n(x)\le N+ O(N/\sqrt{r}) = O(N).
\]
 
\noindent  Next, we analyze $\mathcal T$ beyond the nodes in $S_r(\hat x)$ until both parameters $b(x)$ and $p(x)$ are small.    For each $x \in S_r(\hat x)$, let $S'_r(x)$ be the set of rootmost descendants $y$ of $x$ (including $x$) such that $b(y) \leq c_0\sqrt{r}$ and $p(y) \leq c_0t$. 
We set
\[
S_r':=\bigcup_{x\in S_r} S_r'(x).
\]

\noindent Notice that every node in $S'_r$ satisfies (\ref{eq:leaf-threshold-rdiv}), and therefore, are leaf vertices.    In particular, $
\mathcal R=\{R_y:y\in S_r'\}.$  We now make the following claim as shown in \cite{KMS13}.

\begin{claim}\label{eq:phase2-claim-final}
There is an absolute constant $c_3 > 0$ such that the following holds. For every node $x \in S_r(\hat x)$,   
 $$|S_r'(x)|\le
\max\left\{
1,\,
\frac{b(x)}{c_3\sqrt r}+\frac{p(x)}{c_3t}-12
\right\}.$$
\end{claim}

 \begin{proof}
     
We proceed by induction of the depth of the  subtree rooted at \(x\).  For the base case when $x$ is a leaf, we have \(S_r'(x)=\{x\}\), and the claim is immediate.

For the inductive step, let us assume now that \(x\in S_r(\hat x)\) and at least one of \(b(x)>c_0\sqrt r\) or \(p(x)>c_0t\) holds.
Let \(y_1,\dots,y_k\) be the rootmost descendants \(y\) of \(x\) such that either both $
b(y)\le c_0\sqrt r$ and $p(y)\le c_0t$ hold, or $\ell(y)-\ell(x)=3.$  Clearly we have \(k\le 8\).  Without loss of generality, we can assume that 

\[
|S_r'(y_1)|\ge |S_r'(y_2)|\ge \cdots \ge |S_r'(y_k)|.
\]
Let
\[
q:=\left|\left\{i: |S_r'(y_i)|>1\right\}\right|.
\]

\noindent Iterating
\eqref{eq:b-split-final} and \eqref{eq:p-split-final} through at most three levels gives
\begin{align}
\sum_{i=1}^k b(y_i) &\le b(x)+7c_1\sqrt{c_0r}, \label{eq:sum-b-yi-final}\\
\sum_{i=1}^k p(y_i) &\le p(x). \label{eq:sum-p-yi-final}
\end{align}

\noindent \emph{Case 1.}  Suppose that $q = 0$, that is, we obtain only leaf vertices after three levels down.  Hence, \(|S_r'(x)|\le 8\). Since \(x\) is not a leaf vertex in $\mathcal T$, at least one of \(b(x)>c_0\sqrt r\) or
\(p(x)>c_0t\) holds.  By setting $c_0$ sufficiently large, $c_0 > 10^8c_1$ and $200c_3 = c_0$, we have
\[
|S'_r(x)|
\le 8
<
\frac{c_0}{c_3}-12
\le
\frac{b(x)}{c_3\sqrt r}+\frac{p(x)}{c_3t}-12.
\]

\noindent \emph{Case 2.}  Suppose that $q = 1$.  Then
\[
|S_r'(x)|=|S_r'(y_1)|+(k-1)\le |S_r'(y_1)|+7.
\]
Then we have two cases.   

\smallskip
\noindent{\it Case 2.a.} Suppose there is an ancestor $y$ of $y_1$, that is also a descendant of \(x\), such that the separator chosen for $R_y$ balances the boundary
vertices.  Then
\[
b(y_1)\le \frac34\,b(x)+c_1\sqrt{c_0r}
\qquad \textnormal{and}\qquad
p(y_1)\le p(x).
\]
By the induction hypothesis,
\[
|S_r'(y_1)|\le \frac{b(y_1)}{c_3\sqrt r}+\frac{p(y_1)}{c_3t}-12,
\]

\noindent Since $b(x) \geq c_0\sqrt{r}$, we have

\begin{align*}
|S'_r(x)|
  &\leq |S'_r(y_1)| + 7 \\
  &\le \frac{b(y_1)}{c_3\sqrt r} + \frac{p(y_1)}{c_3 t} +7 \\
  &\le \frac{\frac34\,b(x)+c_1\sqrt{c_0r}}{c_3\sqrt r} + \frac{p(x)}{c_3 t} +7 \\
  &= \frac{b(x)}{c_3\sqrt r}
     - \frac{b(x)}{4c_3\sqrt r}
     + \frac{c_1\sqrt{c_0}}{c_3}
     + \frac{p(x)}{c_3 t}
     +7 \\
  &\le \frac{b(x)}{c_3\sqrt r}
     - \frac{c_0}{4c_3}
     + \frac{c_1\sqrt{c_0}}{c_3}
     + \frac{p(x)}{c_3 t}
     +7 \\
  &\le \frac{b(x)}{c_3\sqrt r} + \frac{p(x)}{c_3 t} - 12.
\end{align*}

\noindent 

\medskip

\smallskip
\noindent{\it Case 2.b:} Suppose there is an ancestor $y$ of $y_1$, that is also a descendant of \(x\), such that the separator chosen for $R_y$ balances the points from $P$.  Then
\[
b(y_1)\le \,b(x)+c_1\sqrt{c_0r}
\qquad \textnormal{and}\qquad
p(y_1)\le \frac34p(x).
\]
By the induction hypothesis,
\[
|S_r'(y_1)|\le \frac{b(y_1)}{c_3\sqrt r}+\frac{p(y_1)}{c_3t}-12,
\]

\noindent Since $p(x) \geq c_0t$,

\begin{align*}
|S'_r(x)|
  &\leq |S'_r(y_1)|  + 7 \\
  &\le \frac{b(y_1)}{c_3\sqrt r} + \frac{p(y_1)}{c_3 t} + 7 \\
  &\le \frac{\,b(x)+c_1\sqrt{c_0r}}{c_3\sqrt r} + \frac{\frac34p(x)}{c_3 t}  + 7\\
  &= \frac{b(x)}{c_3\sqrt r}
     + \frac{c_1\sqrt{c_0}}{c_3}
     + \frac{p(x)}{c_3 t}     - \frac{p(x)}{4c_3t}
     +7 \\
  &\le \frac{b(x)}{c_3\sqrt r}
     + \frac{c_1\sqrt{c_0}}{c_3}
     + \frac{p(x)}{c_3 t}     - \frac{c_0}{4c_3}
     + 7\\
  &\le \frac{b(x)}{c_3\sqrt r} + \frac{p(x)}{c_3 t} - 12.
\end{align*}

\medskip 

\noindent \emph{Case 3.} Finally, assume that \(q\ge 2\). Then
\[
|S_r'(x)|
=
\sum_{i=1}^q |S_r'(y_i)|+(k-q).
\]
Applying the induction hypothesis to \(y_1,\dots,y_q\), and using
\eqref{eq:sum-b-yi-final} and \eqref{eq:sum-p-yi-final}, we obtain

\begin{align*}
|S_r'(x)|
  &\le
  \sum_{i=1}^q
  \left(
    \frac{b(y_i)}{c_3\sqrt r}
    +
    \frac{p(y_i)}{c_3 t}
    - 12
  \right)
  +(k-q) \\
  &\le
  \frac{b(x)+7c_1\sqrt{c_0r}}{c_3\sqrt r}
  +
  \frac{p(x)}{c_3 t}
  - 12q
  +(8-q) \\
  &=
  \frac{b(x)}{c_3\sqrt r}
  +
  \frac{7c_1\sqrt{c_0}}{c_3}
  +
  \frac{p(x)}{c_3 t}
  +
  8
  - 13q \\
  &\le
  \frac{b(x)}{c_3\sqrt r}
  +
  \frac{1400c_1}{\sqrt{c_0}}
  +
  \frac{p(x)}{c_3 t}
  - 18
  \\
  &\le
  \frac{b(x)}{c_3\sqrt r}
  +
  \frac{p(x)}{c_3 t}
  - 12.
\end{align*} \end{proof}

Using Claim \ref{eq:phase2-claim-final}, we have
\begin{align*}
|\mathcal R|
  &= |S_r'|
   = \sum_{x\in S_r(\hat x)} |S_r'(x)| \\
  &\le \sum_{x\in S_r(\hat x)}
     \left(
       1+\frac{b(x)}{c_3\sqrt r}+\frac{p(x)}{c_3t}
     \right) \\
  &= |S_r(\hat x)|
     + \frac{1}{c_3\sqrt r}\sum_{x\in S_r(\hat x)} b(x)
     + \frac{1}{c_3t}\sum_{x\in S_r(\hat x)} p(x).
\end{align*}
As each parent of a node in \(S_r(\hat x)\) has more than \(c_0r\) vertices and the total overcount of vertices among them is $O(N)$ by Claim \ref{eq:L-root-Sr-final}, we have
\[
|S_r(\hat x)|=O\!\left(\frac{N}{r}\right).
\]

\noindent Clearly, $\sum_x p(x) \leq |P|$. Each vertex that belongs to at least two regions in \(S_r(\hat x)\) contributes once to the boundary set of each such region, and contributes one less than this amount to the overcount \(L(\hat x,S_r(\hat x))\). Hence, Claim \ref{eq:L-root-Sr-final} implies that
\[
\sum_{x\in S_r(\hat x)} b(x)\le 2L(\hat x,S_r(\hat x)) = O(N/\sqrt{r}).
\]

\noindent Putting everything together gives $
|\mathcal R| = O\!\left(\frac{N}{r}+\frac{|P|}{t}\right).$  This completes the proof.

\end{proof}

\section{Dense arrangements are locally dense}\label{sec:dense}

Throughout the remainder of the paper, we adopt the following notation. 
Let $P$ be a set of points and $\mathcal{C}$ be a collection of $k$-intersecting curves in the plane. 
For a point $p \in P$, let $|p|$ denote the number of curves in $\mathcal{C}$ incident to $p$. 
Likewise, for a curve $\gamma \in \mathcal{C}$, let $|\gamma|$ denote the number of points of $P$ incident to $\gamma$.

We begin by recalling the following lemmas, all essentially due to Pach and Sharir. The first is a refined version of Theorem~\ref{PS} that follows directly from their proof.

\begin{lemma}[Pach--Sharir]\label{PSlemma}
Let \(P\) be a finite point set and $\mathcal C$ be a finite collection of $k$-intersecting curves.  Suppose there are $m$ crossing points between all pairs of curves in $\mathcal C$ outside of $P.$
Then
\[
I(P,\mathcal C)
\le
O_k\left(
|P|^{\frac{k + 1}{2k + 1}}
m^{\frac{k}{2k + 1}}
+|P|+|\mathcal  C|
\right).
\]
Since we have $m \leq k\binom{|\mathcal C|}{2}$, we obtain
\[
I(P,\mathcal{C}) = O_k\left(
|P|^{\frac{k + 1}{2k + 1}}
|\mathcal C|^{\frac{2k}{2k + 1}}
+|P|+|\mathcal  C|
\right).
\]

\end{lemma}

\noindent An easy consequence of Lemma \ref{PSlemma} is the following.

\begin{lemma}\label{smalldeg}
Let $P$ be a set of points and $\mathcal C$ be a collection of $k$-intersecting curves in the plane. For $\ell \ge 2$,  we have

$$\sum\limits_{p\in P, |p| \geq \ell}|p| \leq  O_k\!\left(
\frac{|\mathcal C|^2}{\ell^{1+1/k}}
+
|\mathcal C|
\right).$$

\end{lemma}

\begin{proof} 
Let $c'_k > 0$ be the implicit constant in Lemma \ref{PSlemma}.  If \(\ell\le 2c'_k\), then the bound is trivial since any two curves have at most $k$ points in common.   Thus we may assume that \(\ell>2c'_k\).
Since every point of \(P_{\ge \ell}\) is incident to at least \(\ell\) curves of \(\mathcal C\), Lemma~\ref{PSlemma} gives
\[
\ell |P_{\ge \ell}|
\le
I
\le
c'_k\!\left(
|P_{\ge \ell}|^{\frac{k+1}{2k+1}}|\mathcal C|^{\frac{2k}{2k+1}}
+
|P_{\ge \ell}|
+
|\mathcal C|
\right),
\]
where 
\[
I=I(P_{\ge \ell},\mathcal C)=\sum\limits_{p\in P, |p| \geq \ell}|p|. 
\] 
Then we have
\[
|P_{\ge \ell}| \leq \frac{I}{\ell}
\qquad\text{and}\qquad
\frac{I}{2} \leq c'_k\!\left(
(\frac{I}{\ell})^{\frac{k+1}{2k+1}}|\mathcal C|^{\frac{2k}{2k+1}}
+
|\mathcal C|
\right). 
\]
If $\frac{I}{4} \le c'_k|\mathcal C|$, then we are done. So let us assume that 
\[
\frac{I}{4} \leq c'_k (\frac{I}{\ell})^{\frac{k+1}{2k+1}}|\mathcal C|^{\frac{2k}{2k+1}}. 
\]
Raising both sides to the power \((2k+1)/k\) and cancelling, we obtain $I = O_k\!\left(\frac{|\mathcal C|^2}{\ell^{1+1/k}}\right).$ Together with the bound from previous case, the statement follows.
\end{proof}

Next, we will need the hypergraph removal lemma. For every fixed $k$-uniform hypergraph $\mathcal H$, any $k$-uniform hypergraph on $n$ vertices containing $o(n^{v(\mathcal H)})$ copies of $\mathcal H$ can be made $\mathcal H$-free by removing $o(n^k)$ edges. This was first proved independently by Gowers~\cite{Gowers2007} and by Nagle, R\"odl, Schacht, and Skokan~\cite{NagleRodlSchacht2006}, building on the hypergraph regularity method (see also \cite{RodlSkokan2004,RodlSchacht2009,Tao2006}). We will use the following standard supersaturation-type reformulation.

\begin{lemma}\label{lem:supersat-hyper}
For every $k\ge 2$, $s\ge k+1$, and $\delta>0$ there exists
$\alpha=\alpha(k,s,\delta)>0$ such that the following holds.
If $\mathcal G$ is a $(k + 1)$-uniform hypergraph on $N$ vertices with at least
$\delta N^{k + 1}$ edge-disjoint copies of $K^{(k + 1)}_s$, then $\mathcal G$
contains at least $\alpha N^s$ copies of $K^{(k + 1)}_s$.
\end{lemma}

We are now ready to prove Theorem~\ref{soly}.  Throughout the proof, we use the \(O(\cdot)\) and \(\Omega(\cdot)\) notation to suppress only those constants that have already been fixed, either earlier in the proof or in previous sections. In particular, any constant that is to be chosen later will always remain explicit until it is determined.

\begin{proof}[Proof of Theorem~\ref{soly}]
Set $T=n^{\frac{1}{2k+1}}$. Let $P$ be a set of $n$ points and $\mathcal{C}$ be a collection of $n$ $k$-intersecting curves in the plane such that $I(P,\mathcal{C}) \geq \varepsilon n^{\frac{3k + 1}{2k + 1}}=\varepsilon \frac{n^2}{T^{k+1}}$, where $n > n_0,T>n^{\frac{1}{2k+1}}_0$ and $n_0$ is a sufficiently large constant that depends only on $\varepsilon, k, s$. We will show that, for sufficiently large \(n_0\), there exist \(s\) points of \(P\) such that every \((k + 1)\)-subset of them lies on a distinct curve of \(\mathcal C\).

Without loss of generality, we can assume that every point in $P$ lies on at least one curve in $\mathcal{C}$, and each curve in $\mathcal{C}$ has at least one point from $P$ on it.  
Moreover, we can assume that the intersection graph of $\mathcal{C}$ is connected, since otherwise, we could pull apart the arrangement and work in each connected component individually.  

Set \(\ell=\frac{c_4 }{\varepsilon^{k/(k+1)}}T^k\), where \(c_4\) is a sufficiently large absolute constant to be chosen momentarily. Let \(Q\subset P\) be the set of points incident to at least \(\ell\) curves of \(\mathcal C\). Choosing \(c_4\) sufficiently large and applying Lemma~\ref{smalldeg}, we may ensure that the number of incidences between \(Q\) and \(\mathcal C\) is at most

$$I(Q,\mathcal C) \leq \sum\limits_{p \in Q}|p| \leq O\left(\frac{n^2}{\ell^{1 + 1/k}} + n\right) \leq \frac{\varepsilon}{2} \frac{n^2}{T^{k+1}}.$$  Thus, for $P'   = P\setminus Q$, the number of incidences between $P'$ and $\mathcal C$ is at least $\frac{\varepsilon}{2} \frac{n^2}{T^{k+1}}$.  Moreover, each point $p \in P'$ is incident to at most $\ell$ curves in $\mathcal{C}$.  Let us now restrict our attention to the arrangement $(P', \mathcal C)$.  By locally perturbing the curves in $\mathcal C$, without loss of generality, we can assume that no three curves in $\mathcal{C}$ have a common point outside of $P'$.
Let $m$ denote the number of crossing points among the curves in $\mathcal{C}$ outside of $P'$.  
By Lemma~\ref{PSlemma}, we have $m=\Omega_k\!\left(\varepsilon^{\frac{2k+1}{k}}n^2\right)$.  

From the arrangement $(P',\mathcal{C})$, consider the planar graph $G_0$ embedded in the plane whose vertex set consists of the $m$ crossing points together with the points of $P'$, and whose edges are consecutive vertices along a curve in $\mathcal{C}$.  
We then modify $G_0$ as follows.  For each point $p \in P'$, we add $w = \varepsilon^{(2k + 1)/k}T^{2k+1}/\ell=\varepsilon^{(2k + 1)/k}n/\ell$ nested cycles centered at $p$, and very close to $p$, such that each such cycle has length $2|p|$ and the resulting graph remains planar.   See Figure~\ref{X}.

\begin{figure}
    \centering
    \includegraphics[scale=.35]{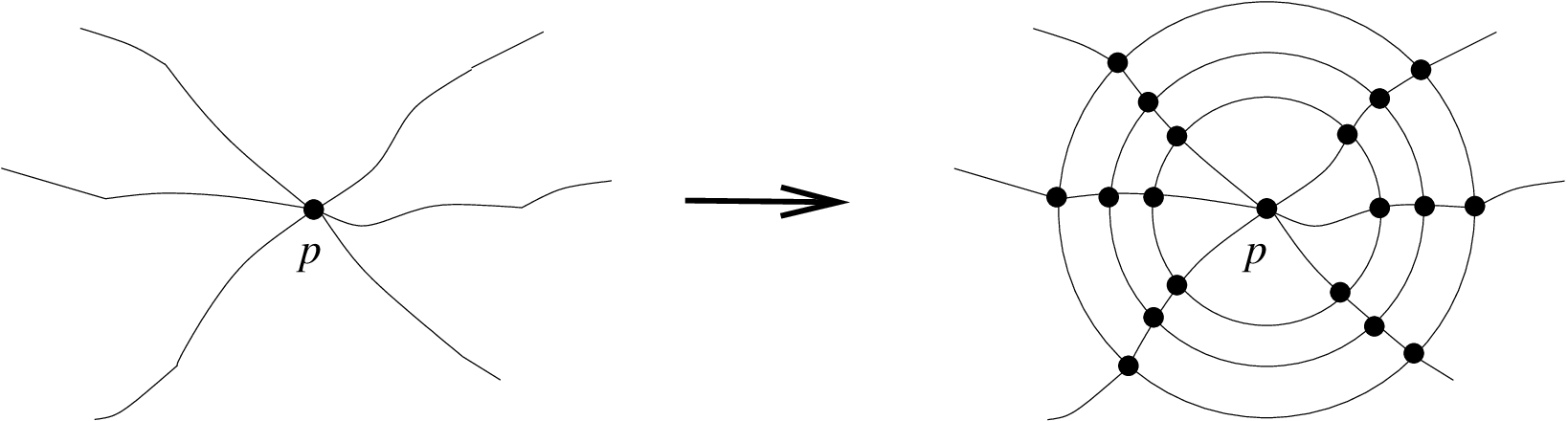}
    \caption{The local modification near $p$.}
    \label{X}
\end{figure}

After performing this operation for every $p \in P'$, and since $|p|\leq \ell$,  we will have added at most an additional $2\varepsilon^{\frac{2k+1}{k}}n^2$ vertices to $G_0$. We then triangulate every face of size greater than three in the resulting graph, obtaining a planar graph \(G=(V,E,F)\) embedded in the plane in which every face has size \(2\) or \(3\), and
\[
|V| = m + 2\varepsilon^{\frac{2k+1}{k}}n^2 + |P'| = O(m).
\]

\noindent  Note that $ m \leq kn^2$.  We now apply Theorem~\ref{rdiv} to $G$ and the  point set $P'\subset V(G)$, with parameters

\[
r = c_5\frac{k^2s^2}{\varepsilon^2}T^{2k+2}, 
\qquad
t =  c_5\frac{ks^2}{\varepsilon^2}T,
\]

\noindent where $c_5$ is a large absolute constant that will be chosen later.  Hence, we obtain a division $\mathcal R$ such that

$$|\mathcal R| = O\left( \frac{|V|}{r } + \frac{n}{t}\right) =  
O\!\left(
\frac{\varepsilon^2}{c_5ks^2}
\frac{n^2}{T^{2k+2}}
\right).$$

\noindent and for each region $R_i \in \mathcal R$ we have

\[
|V(R)|\le r,\qquad |b(R)|=O(\sqrt r), \qquad 
  |(V(R)\setminus b(R))\cap P'| \le t.
\]

\noindent  Moreover, there is a collection of simple cycles (simple closed curves in the plane) $\mathcal S$ such that for each boundary vertex $v$, there is a simple cycle $C_x \in \mathcal{S}$ that contains $v$, such that either

\begin{enumerate}
     
    \item the number of vertices of $G$ inside or on $C_x$ is at least $r/8$, or

    \item the number of boundary vertices inside or on $C_x$ is at least $\sqrt{r}/8$, or

    \item the number of points from $P$ inside or on $C_x$ is at least $t/8.$
 \end{enumerate}

Let \(Q' \subset P'\) be the set of boundary vertices of the division \(\mathcal R\).  That is, \(p\in Q'\) implies that \(p\) lies on some cycle \(C_i\in\mathcal S\). Fix \(p\in Q'\), and let \(Z_1,\dots,Z_w\) be the $w$ nested cycles constructed around $p$.  The cycles are ordered from inside to outside, so that \(Z_1\) is the innermost and \(Z_w\) the outermost. Let \(G_p\) be the subgraph of \(G\) induced by \(p\) together with the vertices of \(Z_1,\dots,Z_w\).  See Figure \ref{X}.

\begin{observation}
If $p \in Q'$, then at least  $c_6T^{k + 1}$ vertices from $G_p$ are also boundary vertices, where $c_6 = c_6(k,\varepsilon)$.
\end{observation}

\begin{proof}
 
The proof falls into two cases.

\medskip
\noindent\emph{Case 1.} Suppose that there exists a node \(x\in\mathcal T\) such that the separating cycle $C_x$ lies inside of $Z_i$ for some $i\geq 1$. That is, the simple closed curve in the plane corresponding to $C_x$ lies inside the simple closed curve corresponding to $Z_i$.  Hence, $V(C_x) \subset V(G_p)$.   Since \(G_p\) has \(1+2w|p|\) vertices and \(|p|\le \ell\), we have
\[
|V(G_p)| \le 1+2w\ell = O\!\left(\varepsilon^{(2k+1)/k}T^{2k+1}\right).
\]
For \(T\) sufficiently large,
\[
|V(G_p)|<\frac r8,
\qquad\text{and}\qquad
|V(G_p)\cap P'|=1<\frac t8.
\]
Therefore, property (2) in Theorem \ref{rdiv} must hold, and the number of boundary vertices inside or on \(C_x\) is at least $\frac{\sqrt r}{8}
=
\frac{\sqrt{c_5}\,ks}{8\varepsilon}\,
T^{k+1}.$  Since \(C_x\) lies inside $Z_i$, all of these boundary vertices lie in \(G_p\). Thus for $c_6 > 0$ sufficiently small, \(G_p\) contains at least $c_6 T^{k+1}$ boundary vertices.

\medskip
\noindent\emph{Case 2.} Suppose we are not in Case 1.  Let \(x\in\mathcal T\) be the highest node such that \(p\in V(C_x)\), and let \(R_x\) be the corresponding region. 

Let \(i\) be the largest index such that the cycles \(Z_1,\dots,Z_i\) are fully contained in \(R_x\). (Possibly \(i=0\).) Then for each \(j>i\), the cycle \(Z_j\) is not fully contained in \(R_x\). We first claim that for every \(j>i\), the cycle \(Z_j\) already contains a boundary vertex. Indeed, \(Z_j\) is fully present in the root region, but not fully present in \(R_x\). Hence along the root-to-\(x\) path in \(\mathcal T\), there is a first node \(y_j\) such that \(Z_j\) is not fully contained in \(R_{y_j}\). Let \(z_j\) be the parent of \(y_j\). Then \(Z_j\) is fully contained in \(R_{z_j}\), but not in \(R_{y_j}\). Therefore the separator cycle \(C_{z_j}\) must contain a vertex of \(Z_j\). Such a vertex is a boundary vertex. Since the cycles \(Z_j\) are pairwise disjoint, these boundary vertices are distinct for different \(j>i\). Thus we obtain at least \(w-i\) boundary vertices from the cycles \(Z_{i+1},\dots,Z_w\).

Next consider the cycles \(Z_1,\dots,Z_i\). Since each such cycle is fully contained in \(R_x\), while \(C_x\) contains a point outside of $Z_i$, the Jordan curve \(C_x\) must meet each \(Z_1,\ldots, Z_i\). Hence for each \(1\le j\le i\), the cycle \(Z_j\) contains a boundary vertex lying on \(C_x\). Again these vertices are distinct because the cycles \(Z_1,\dots,Z_i\) are pairwise disjoint.

Combining the two parts, we find that \(G_p\) contains at least
\[
(w-i)+i=w
\]
distinct boundary vertices. Since
\[
w=\frac{\varepsilon^{(2k+1)/k}T^{2k+1}}{\ell}
\geq c_6 T^{k+1},
\]
by setting \(c_6=c_6(k,\varepsilon)\) sufficiently small, the proof is complete.
\end{proof}

By the Observation above, the total number of boundary vertices is at least $|Q'|c_6T^{k + 1}$, which implies

$$|Q'|c_6T^{k + 1} \leq O( |\mathcal R|\sqrt{r}) \leq O\!\left(
\frac{\varepsilon}{\sqrt{c_5}\,s}
\frac{n^2}{T^{k+1}}
\right).
$$

\noindent  This implies

\[
|Q'|
\le
O\!\left(
\frac{\varepsilon}{c_6\sqrt{c_5}\,s}
\frac{n^2}{T^{2k+2}}
\right)=O\!\left(
\frac{\varepsilon}{c_6\sqrt{c_5}\,s}
\frac{n}{T}
\right).
\]

\noindent By Lemma \ref{PSlemma}, we have

$$ I(Q',\mathcal C) = O_k\left(
|Q'|^{\frac{k + 1}{2k + 1}}
|\mathcal C|^{\frac{2k}{2k + 1}}
+|Q'|+|\mathcal  C|
\right) \leq  O_k\!\left(
\left(\frac{\varepsilon}{c_6\sqrt{c_5}\,s}\right)^{\frac{k+1}{2k+1}}
\frac{n^2}{T^{k+1}}\frac{1}{T^{\frac{k+1}{2k+1}}}
+
\frac{\varepsilon }{c_6\sqrt{c_5}\,s}\frac{n^2}{T^{2k+2}}
+
n
\right).$$

\noindent Hence for $n_0 = n_0(c_5,c_6,s, \varepsilon, k)$ sufficiently large, for all $n > n_0$ with $T>n^{\frac{1}{2k+1}}_0$ we have

\[
I(Q',\mathcal C) \leq \frac{\varepsilon}{4}\frac{n^2}{T^{k+1}}.\]

Hence, there are at least $\frac{\varepsilon}{4}\frac{n^2}{T^{k+1}}$ incidences between $P'\setminus Q'$ and $\mathcal C$.  The division $\mathcal R$ gives a natural partition of $P'\setminus Q' = \bigcup_i P_i$, where 

$$P_i = \{p \in P'\setminus Q': p \in V(R_i)\}.$$

\noindent  Note that all points in $P_i$ are not boundary vertices.   By Theorem~\ref{rdiv}, we have that $|P_i|  \leq t =  c_5\frac{ks^2}{\varepsilon^2}T$.

Let $\Gamma$ be the incidence graph between $P'\setminus Q'$ and $\mathcal C$. Consider a curve $\gamma \in \mathcal{C}$.  Let $b(\gamma)$ be the number of boundary vertices of $G$ on $\gamma$ with respect to $\mathcal{R}$.  Along each maximal subarc of \(\gamma\) whose interior contains no boundary vertex, the points from a fixed part \(P_i\) appear consecutively, we partition them into consecutive blocks of size \(s\), and discard the remaining incidences from the subarc which has at most $s-1$ points. Since each such subarc lies between two consecutive boundary points on \(\gamma\), the total number of deleted incidences on \(\gamma\) is at most \(s\,(b(\gamma)+1)\).  Each boundary vertex either is a point in $Q'$ with degree at most $2\ell$, or is a vertex with degree 4. Hence, after performing this operation for each $\gamma \in \mathcal C$, we will delete at most
\[
\sum\limits_{\gamma \in \mathcal C} s(b(\gamma)  +1)  = O\!\left(s\sqrt r\,|\mathcal R|+s\ell|Q'|+s|\mathcal C|\right)
=
O\!\left(
\frac{\varepsilon}{\sqrt{c_5}}
\frac{n^2}{T^{k+1}}
+
\frac{\varepsilon^{\frac{1}{k+1}}}{\sqrt{c_5}}
\frac{n^2}{T^{k+2}}+n
\right).
\]
 
\noindent incidences.  Indeed, there are at most $O(\sqrt{r}|\mathcal R|)$ boundary points, and each point $p \in Q'$ has at most $\ell$ curves going through it.   By setting $c_5$ to be a sufficiently large absolute constant, and for $n,T$ sufficiently large, the quantity above is less than $\frac{\varepsilon}{8} \frac{n^2}{T^{k+1}}$. Thus, at least $\frac{\varepsilon}{8} \frac{n^2}{T^{k+1}}$ edges (incidences) remain in $\Gamma$. 

By the pigeonhole principle, there is a part $P_i \subset P'\setminus Q'$ such that the number of edges between $P_i$ and $\mathcal C$ in $\Gamma$ is at least
\[
\frac{\varepsilon \frac{n^2}{T^{k+1}}}{8|\mathcal R|}
=
\Omega\!\left(
\frac{ks^2}{\varepsilon}
T^{k+1}
\right),
\]

\noindent and
\[
 |P_i|  \leq t \leq  c_5\frac{ks^2}{\varepsilon^2}T.
\]

Moreover, if a curve $\gamma \in \mathcal C$ is incident to a point $p \in P_i$, and $(p,\gamma)$ is a surviving edge in $\Gamma$, then $\gamma$ must be incident to at least $s$ vertices in $P_i$.  Let $\mathcal C_0$ be the curves in $\mathcal C$ that are incident to at least $s$ points in $P_i$. By Lemma~\ref{PSlemma}, we have
\[
\Omega\!\left(
\frac{ks^2}{\varepsilon}
T^{k+1}
\right)  = O_k\left(
|P_i|^{\frac{k + 1}{2k + 1}}
|\mathcal C_0|^{\frac{2k}{2k + 1}}
+|P_i|+|\mathcal  C_0|
\right).
\]

\noindent Since $|P_i|=O_{\varepsilon,k}\!\left(T\right)$, the term $|P_i|$ is negligible for large $T$. Therefore

\[
|\mathcal C_0|
=
\Omega_k\!\left(
 \sqrt{k}\,s\,\varepsilon^{1/(2k)} 
T^{k+1}
\right).
\]

Let $H=(P_i,E)$ be the auxiliary $(k + 1)$-uniform hypergraph whose vertex set is $P_i$. For each $\gamma \in \mathcal C_0$, consider the points of $P_i$ that appear along $\gamma$ from one end to the other. We then create a $K^{(k+1)}_s$ on the first $s$ points along $\gamma$, then create another $K^{(k+1)}_s$ on the next set of $s$ points along $\gamma$, and repeat this process until there are fewer than $s$ points left on $\gamma$. Thus each $\gamma$ gives rise to at least one copy of $K^{(k+1)}_s$.  Moreover, since no two curves in $\mathcal C_0$ have $k + 1$ points in common, the corresponding copies of $K_s^{(k + 1)}$ are edge disjoint. Since
\[
|\mathcal C_0| \geq \Omega_k\!\left(
 \sqrt{k}\,s\,\varepsilon^{1/(2k)} 
T^{k+1}
\right)
\quad\text{and}\quad
N = |P_i|  \leq c_5\frac{ks^2}{\varepsilon^2}T,
\]
we have $|\mathcal C_0| \geq \delta N^{k + 1}$ for some constant $\delta = \delta(\varepsilon,k,s) > 0$. Thus $H$ is a $(k +1)$-uniform hypergraph on $N$ vertices that contains at least $\delta N^{k + 1}$ edge-disjoint copies of $K^{(k+1)}_s$. By Lemma~\ref{lem:supersat-hyper}, it follows that $H$ contains at least $\alpha N^s$ copies of $K^{(k+1)}_s$, where $\alpha=\alpha(k,\varepsilon,s)$.

We now bound the number of $K^{(k+1)}_s$ whose $s$ vertices contain $k + 2$ points from $P_i$ lying on a single curve $\gamma\in\mathcal C_0$. These $k + 2$ points must lie in a single block of $s$ consecutive points along that curve.  There are at most $sN^{k + 1}$ such $(k +2)$-tuples, since no two curves have $k + 1$ points in common. Once the \(k+2\) vertices lying on a single curve are chosen, the remaining \(s-k-2\) vertices may be chosen arbitrarily from \(P_i\). Thus each such $(k+2)$-tuple gives rise to at most \(O(N^{\,s-k-2})\) bad copies. Since there are at most \(N^{k+1}\) such tuples, the total number of such copies of $K_s^{(k + 1)}$ is
\[
O\!\left(sN^{k+1}\cdot N^{\,s-k-2}\right)=O(sN^{s-1}).
\]

 Finally, by setting $n > n_0$ sufficiently large, the number of such bad copies is at most
\[
csN^{s-1} < \alpha N^s,
\]
since \(N\) will be large. Hence, there exists a copy of \(K^{(k + 1)}_s\) in $H$ in which no $k + 2$ vertices lie on the same curve of \(\mathcal C_0\). Equivalently, there are \(s\) points of \(P_i\) such that each $(k + 1)$-tuple  is joined by a distinct curve from \(\mathcal C_0\), yielding \(\binom{s}{k + 1}\) distinct curves determined by these points.
\end{proof}

\medskip

Let us remark that the above proof goes through verbatim to the non-diagonal case, with a change of parameter 
\[
T=\left( \frac{|\mathcal{C}|^2}{|P|}\right)^{\frac{1}{2k+1}}
\]
and extends to the following theorem. 
\begin{theorem}
Let \(s>k+1\ge 2\) be fixed integers. Let \(P\) be a set of points in the plane and \(\mathcal C\) a family of simple \(k\)-intersecting curves. Suppose that \(P\) does not contain \(s\) points such that every \((k+1)\)-tuple among them is contained in a distinct curve of \(\mathcal C\). Then for every \(\epsilon>0\),
\[
I(P,\mathcal C)< \epsilon\, |P|^{\frac{k+1}{2k+1}}\,|\mathcal{C}|^{\frac{2k}{2k+1}},
\]
provided that \(\frac{|P|^{k+1}}{|\mathcal{C}|}\), \(|\mathcal{C}|\), and \(\frac{|\mathcal{C}|^2}{|P|}\) are sufficiently large in terms of \(\epsilon\).
\end{theorem}

\section{Proof of Theorem \ref{thm:lower}}\label{sec:lower}

We use a standard point--line construction for the Szemerédi--Trotter theorem (see \cite{SuT}).  Let
\[
P=\{(a,b)\in \mathbb{N}^2: a<n^{1/3},\ b<n^{2/3}\}
\]
and
\[
\mathcal L=\{y=ax+b: a,b\in \mathbb{N},\ a<n^{1/3},\ b<n^{2/3}\}.
\]
Then \(|P|=|\mathcal L|=n\), and
\[
I(P,\mathcal L)=\Theta(n^{4/3}).
\]

We then define a graph \(G\) on vertex set \(P\), and whose edges are pairs of points in $P$ that lie on a common line of \(\mathcal L\).  We will use the following lemma of Suk and Tomon.

\begin{lemma}[{\cite[Claim 16]{SuT}}]
For the graph \(G\) arising from the standard construction above, every vertex has degree at most $n^{2/3}$, and every pair of distinct vertices has at most $n^{1/3+o(1)}$ common neighbors.
\end{lemma}

\begin{proof}[Proof of Theorem \ref{thm:lower}]
Any set of \(s\) points of \(P\) such that every pair lies on a distinct line of \(\mathcal L\) spans a copy of \(K_s\) in \(G\). We bound the number of such copies as follows. First choose an edge \(uv\) of \(G\). Since every vertex of \(G\) has degree at most \(n^{2/3}\), there are at most
\[
n\cdot n^{2/3}=n^{5/3}
\]
choices for \(u\) and \(v\). Once \(u\) and \(v\) are fixed, every remaining vertex of the clique must belong to the common neighborhood of \(u\) and \(v\), which has size at most \(n^{1/3+o(1)}\) by the lemma. Hence the number of such \(s\)-point sets is at most
\[
n^{5/3}\left(n^{1/3+o(1)}\right)^{s-2}
=
n^{1+s/3+o(1)}.
\]

Now independently select each incidence of \((P,\mathcal L)\) with probability \(p\), where \(p\) will be chosen later. Thus the expected number of selected incidences is
\[
\Theta(n^{4/3}p).
\]

For every unselected incidence \((q,\ell)\), we modify \(\ell\) inside a sufficiently small neighborhood of \(q\) so that the resulting curve avoids \(q\) but agrees with \(\ell\) outside that neighborhood. Choosing these neighborhoods iteratively, making them sufficiently small each time, no new incidences with points of \(P\) are created, and every pair of resulting curves still intersects at most once.

Fix a set \(S\subset P\) of size \(s\) such that every pair of points of \(S\) lies on a distinct line of \(\mathcal L\). For each pair \(\{u,v\}\subset S\), let \(\ell_{uv}\) denote the unique line of \(\mathcal L\) containing \(u\) and \(v\). In order for \(S\) to survive after the perturbation, both incidences \((u,\ell_{uv})\) and \((v,\ell_{uv})\) must be selected for every pair \(\{u,v\}\subset S\). Since these are \(2\binom{s}{2}\) distinct incidences, it follows that
\[
\Pr(S\text{ survives})=p^{2\binom{s}{2}}.
\]

\noindent Therefore the expected number of surviving bad \(s\)-tuples is at most
\[
n^{1+s/3+o(1)}p^{2\binom{s}{2}}.
\]

\noindent Choose \(p\) so that
\[
n^{4/3}p/2
=
n^{1+s/3+o(1)}p^{2\binom{s}{2}}.
\]
Since
\[
2\binom{s}{2}-1=s(s-1)-1,
\]
this gives
\[
p
=
n^{-\frac{s-1}{3(s^2-s-1)}-o(1)}.
\]
With this choice, the expected number of surviving bad \(s\)-tuples is at most half the expected number of selected incidences. Therefore, by the deletion method, there exists a choice of selected incidences such that after deleting one selected incidence from each surviving bad configuration, the number of remaining incidences is at least
\[
n^{4/3}p/2
=
n^{\frac43-\frac{s-1}{3(s^2-s-1)}-o(1)}.
\]

Performing the corresponding local perturbations, we obtain a set of \(n\) points and \(n\) pseudo-segments in the plane with no \(s\) points such that every pair lies on a distinct curve, and with at least
\[
n^{\frac43-\frac{s-1}{3(s^2-s-1)}-o(1)}
\]
incidences. 
\end{proof}

\section{Concluding remarks}

A longstanding open problem is to determine the correct extremal exponent for incidences between \(n\) points and \(n\) \(k\)-intersecting curves when \(k\ge 2\). As mentioned in the introduction, the best known lower bound is still the classical Szemer\'edi--Trotter construction, which gives \(\Omega(n^{4/3})\), while the best general upper bound is the Pach--Sharir bound
\[
O\!\left(n^{\frac{3k+1}{2k+1}}\right).
\]
For every fixed \(k\ge 2\), there remains a substantial gap between these bounds. It would be very interesting to determine whether the correct asymptotic behavior is closer to the lower bound or the upper bound. On the other hand, the exponent in the Pach--Sharir bound cannot be improved in the unbalanced regimes \( |P|=\Omega(|\mathcal{C}|^{2}) \) or \( |\mathcal C|=\Omega(|P|^{k+1}) \). 

Another interesting direction is to determine whether one can obtain a polynomial improvement in Theorem~1.2. Since our argument relies on the hypergraph removal lemma, it yields only a qualitative \(o(\cdot)\)-improvement over the Pach--Sharir bound. Even in the special case of point-line arrangements, proving a polynomial improvement beyond Solymosi's result~\cite{SolymosiLocallyDense} remains a challenging problem.

\end{document}